\documentclass[final,1p,times]{elsarticle}
\usepackage[utf8]{inputenc}
\usepackage{commath}
\usepackage{graphicx}
\usepackage{bm}
\usepackage{amssymb}
\usepackage{amsmath}
\usepackage{amsthm}
\usepackage{array}
\usepackage{tabularx}
\usepackage{subfloat}
\usepackage{subfig}
\usepackage[export]{adjustbox}
\usepackage{tikz}
\usepackage[most]{tcolorbox}
\usepackage[colorlinks]{hyperref}
\usepackage{cleveref}
\usepackage{physics} % Used for residual symbol
\usetikzlibrary{positioning}
\usetikzlibrary{snakes}

\newtheorem{thm}{Theorem}
\newtheorem{rmk}[thm]{Remark}

\usepackage{showlabels}

\newcommand{\ie}{{\it i.e.\ }}
\newcommand{\eg}{{\it e.g.\ }}

\renewcommand{\d}{\mathrm{d}} 
\newcommand{\bigO}{{\mathcal O}}
\newcommand{\be}{\begin{equation}}
\newcommand{\ee}{\end{equation}}
\renewcommand{\S}{{\mathcal S}}
\newcommand{\D}{{\mathcal D}}
\newcommand{\phip}{{\Phi_{\mathrm{p}}}}
\newcommand{\phipf}{{\Phi_{\mathrm{p}, \text{far}}}}
\newcommand{\phipn}{{\Phi_{\mathrm{p}, \text{near}}}}
\renewcommand{\vec}[1]{\boldsymbol{\mathbf{#1}}}
\crefname{equation}{}{}

\newcommand{\xx}{\mathbf{x}}
\newcommand{\yy}{\mathbf{y}}
\newcommand{\kk}{\mathbf{k}}
\newcommand{\nn}{\mathbf{n}}
\newcommand{\lv}{\mathbf{d}}        % lattice vec
\newcommand{\R}{\mathbb{R}}
\newcommand{\Z}{\mathbb{Z}}
\newcommand{\C}{\mathbb{C}}
\DeclareMathOperator{\im}{Im}
\DeclareMathOperator{\re}{Re}
\DeclareMathOperator{\Nul}{Nul}
\DeclareMathOperator{\Dim}{dim}
\newcommand{\trp}{\mathrm{tr}} % subscript for trapped
\newcommand{\res}{\mathrm{res}} % subscript for calculating residual
\newcommand{\pO}{\partial\Omega} % bdry
\newcommand{\bea}{\begin{eqnarray}}
\newcommand{\eea}{\end{eqnarray}}

\journal{Journal of Computational Physics}

\begin{document}

\begin{frontmatter}

\title{Trapped acoustic waves and raindrops:
  high-order accurate integral equation method
  for localized excitation of a periodic staircase}

\author[1]{Fruzsina J.\ Agocs\corref{cor1}} 
\cortext[cor1]{corresponding author.}
    \ead{fagocs@flatironinstitute.org}
\author[1]{Alex H.\ Barnett}
    \ead{abarnett@flatironinstitute.org}
\affiliation[1]{organization = {Center for Computational Mathematics, Flatiron Institute},
                    addressline = {162 Fifth Avenue},
                    city = {New York},
                    state = {NY},
                    country = {USA}}

\begin{abstract}
  We present a high-order boundary integral equation
  (BIE)
  method for the frequency-domain acoustic scattering of a point source by a singly-periodic, infinite, corrugated boundary. 
We apply it to the accurate numerical study of acoustic radiation in the neighborhood of a sound-hard two-dimensional staircase modeled after the \textit{El Castillo} pyramid.
Such staircases support trapped 
waves which travel along the surface and decay exponentially away from it.
We use the array scanning method (Floquet--Bloch transform) to recover
the scattered field as an integral over the family of
quasiperiodic solutions parameterized by on-surface wavenumber.
Each such BIE solution requires 
the quasiperiodic Green's function, which we evaluate using an efficient integral representation of lattice sum coefficients.
We avoid the singularities and branch cuts present in the array scanning integral by
complex contour deformation.
For each frequency, this enables a solution accurate to around 10 digits in a couple of seconds.
We propose a residue method to extract the limiting powers carried by trapped modes far from the source.
Finally, by computing the trapped mode dispersion relation,
we use a simple ray model to explain an observed acoustic ``raindrop'' effect
(chirp-like time-domain response).
\end{abstract}

\end{frontmatter}

\section{Introduction \label{intro}}

Periodic surface geometries
have long been exploited to manipulate electromagnetic
and acoustic waves. Examples on small and large length-scales include photonic
crystals \cite{busch2007periodic,jobook}, acoustic metamaterials
\cite{chen2017general,ji2022recent}, diffraction gratings
\cite{lin2019integral}, antennae \cite{munk1979plane,he2007radiation}, anechoic
chambers and materials \cite{herrero2020sound}, and amphitheatres \cite{declercq2007acoustic}.
However, the accurate numerical solution of radiation
near such geometries faces several challenges.
The domain---a perturbation of the upper half-space---%
is unbounded vertically,
thus its truncation must incorporate
the correct upward-going radiation conditions.
The possibility of coupling to waves trapped \textit{along} the surface,
which in two dimensions (2D) do not decay at all,
means that $\bigO(1)$ reflection errors would result by naive
truncation to any finite number of unit cell periods.
These trapped waves may also cause resonances with high
parameter sensitivity \cite{shipmanreview}.
Having non-periodic excitation breaks the periodicity of the problem, which
makes periodization impossible at first glance. However, it is possible to
deconstruct the single point-source solution into sets of quasiperiodic
solutions (the array scanning or Floquet--Bloch method).
As in any wave problem,
high frequencies may demand a large discretization density \cite{chandler2012numerical}.
Finally, the staircase model that we focus on here introduces corner
singularities that must be addressed in any high-order solver.
Since solvers are
often part of a design optimization loop to tune
material or shape parameters, they must be robust and efficient.
However, the above difficulties mean that nonconvergent Rayleigh expansions
\cite{richards2018acoustic}, or ray asymptotic methods
\cite{tsingos2007extending} are often used in acoustic settings.

The purpose of this paper is twofold:
1) We present a high-order numerical
boundary integral (BIE) method for 2D acoustic scattering
of a point source from a singly-periodic geometry, in particular
combining it with a high-order accurate array scanning method
via contour deformation.
2) We apply this to the accurate numerical study of acoustic radiation in the neighborhood of a sound-hard 2D staircase model, 
in order to understand time- and frequency-domain phenomena
due to nearby point-source excitation.
For the time-domain we compute the dispersion relation of trapped modes
and use it to explain an observed chirp effect \cite{calleja,hellerbook};
for the frequency domain we propose a method
to extract the amplitudes of the left- and right-going trapped modes,
using residues in the complexified on-surface wavenumber.
Architecturally, solid staircases with sound-hard surfaces are
clearly very common;
yet, we have not found an accurate numerical study of acoustic trapping
and guiding in their vicinity.
By combining high-order corner quadratures and lattice sums,
we show that a BIE can achieve around 10 digits of accuracy with
solution times of a couple of seconds per excitation frequency.

We consider a 2D model for
linearized acoustics in a simply-connected region $\Omega\subset \R^2$
containing a constant-density gas with constant sound speed $c>0$
\cite{coltonkress,howe1998acoustics},
\cite[Sec.~3.1]{coltonkress_scatt}.
This region lies above a connected corrugated boundary $\pO$ extending
with spatial period $d$ in the $x_1$ direction,
and is unbounded in the positive $x_2$ direction;
see \cref{stair-geometry}.
The boundary has a maximum and minimum $x_2$ coordinate.
Our paradigm example will be the slope-$\pi/4$ right-angle staircase shown
in the figure, whose repeating element (unit cell) comprises
two equal-length line segments at an angle $\pi/2$ to each other.
We write $\xx=(x_1,x_2)$.
This models a 3D situation in which both geometry and source
are invariant in the 3rd (out of plane) direction.

The tools we present mostly concern frequency domain solutions,
but we will also use them to understand a chirp phenomenon for the
following time domain problem.
Consider an impulse excitation by a point source at time $t=0$
and location $\xx_0\in\Omega$; this is a good model
for a clap or footstep in the acoustic application.
The acoustic pressure $U(t,\xx)$ then obeys the wave equation
with a source,
\be
\frac{\partial^2 U(t, \vec{x})}{\partial t^2}
- \Delta U(t, \vec{x})  =
\delta(t)\delta(\xx-\xx_0),
\qquad t\in\R, \quad \xx\in\Omega,
\label{WE}
\ee
where $\Delta  = \nabla \cdot \nabla$ is the spatial Laplacian,
and by a rescaling of time $t$ we set the sound speed to $c=1$.
We assume quiescence before the excitation: $U\equiv 0$ for $t<0$.
\footnote{Note that \cref{WE} could thus be rephrased as the homogeneous
wave equation in $t>0$ with initial conditions $U(0,\cdot)\equiv0$,
$U_t(0,\cdot) = \delta_{\xx_0}$.}
The sound-hard (Neumann) boundary condition is
\be
U_n(t,\xx) = 0, \qquad  t\in\R, \quad \xx \in \partial\Omega,
\label{WEBC}
\ee
where $U_n := \nn \cdot\nabla U$ is the normal derivative, $\nn$
being the unit boundary normal vector pointing into $\Omega$.
This arises physically since the normal component of the fluid velocity
vanishes; it is a good approximation to an air-solid interface
\cite[Ch.\ 1]{kaltenbacher2018computational}.

Taking the Fourier transform of \cref{WE,WEBC} with respect to $t$
gives the main focus of this paper, the inhomogeneous Helmholtz
Neumann boundary value problem (BVP),
\bea
-(\Delta + \omega^2) u  & =& \delta_{\xx_0}   \qquad \mbox{ in } \Omega,
\label{helmnp}
\\
u_n  &=& 0 \qquad \mbox{ on } \partial\Omega,
\label{neubc}
\eea
and radiation conditions explained below.
Here $\omega\in\R$ is the frequency, and the wavelength is $2\pi/\omega$.
From the $\omega$-dependence of its solution $u$
one may understand features of the above time-domain solution $U$.

\begin{figure}[tbp]
\centering
    \subfloat{\includegraphics[width = 0.27\textwidth, valign = c]{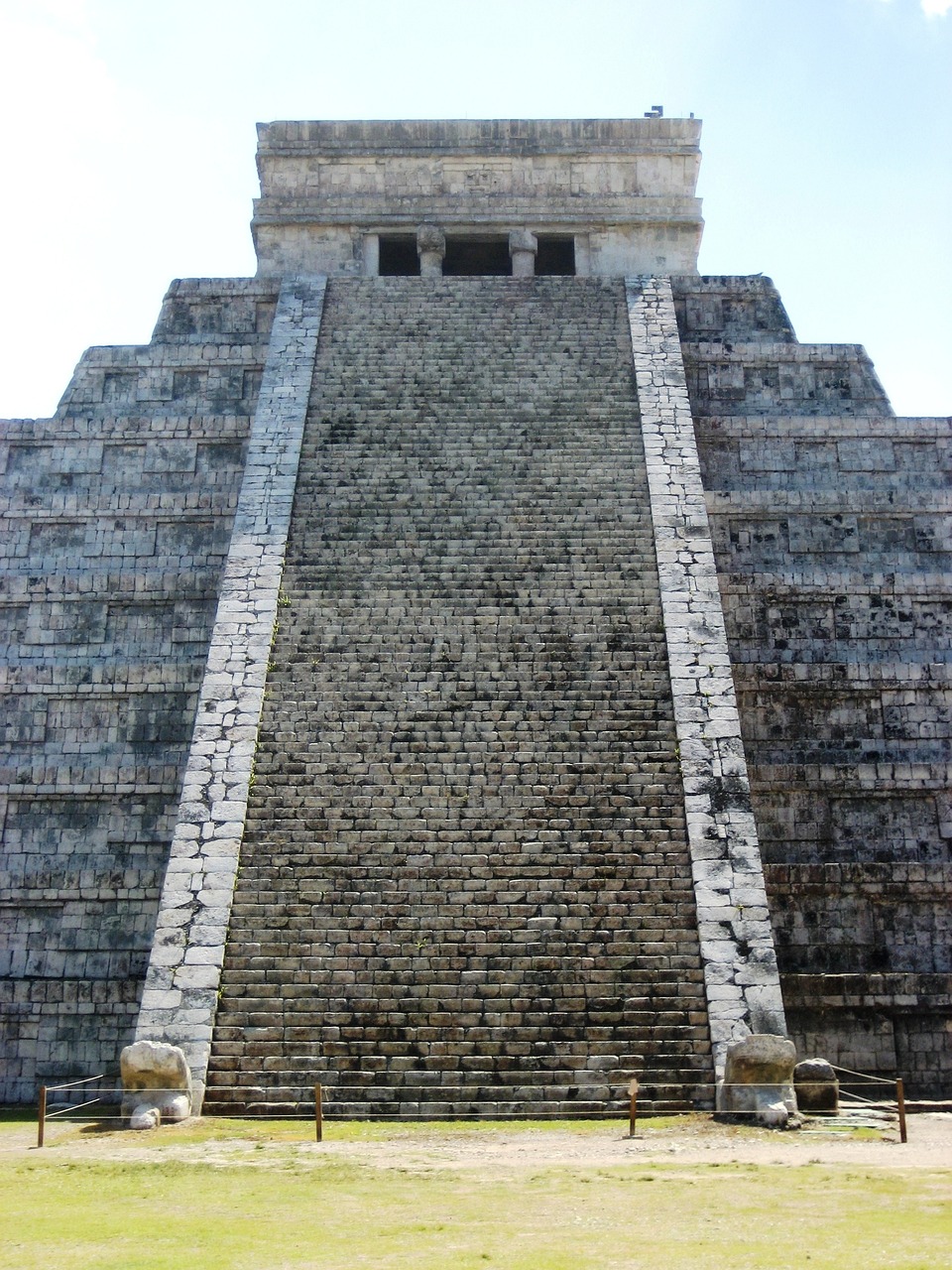}}
    \hfill
    \subfloat{\includegraphics[width = 0.7\textwidth, valign = c]{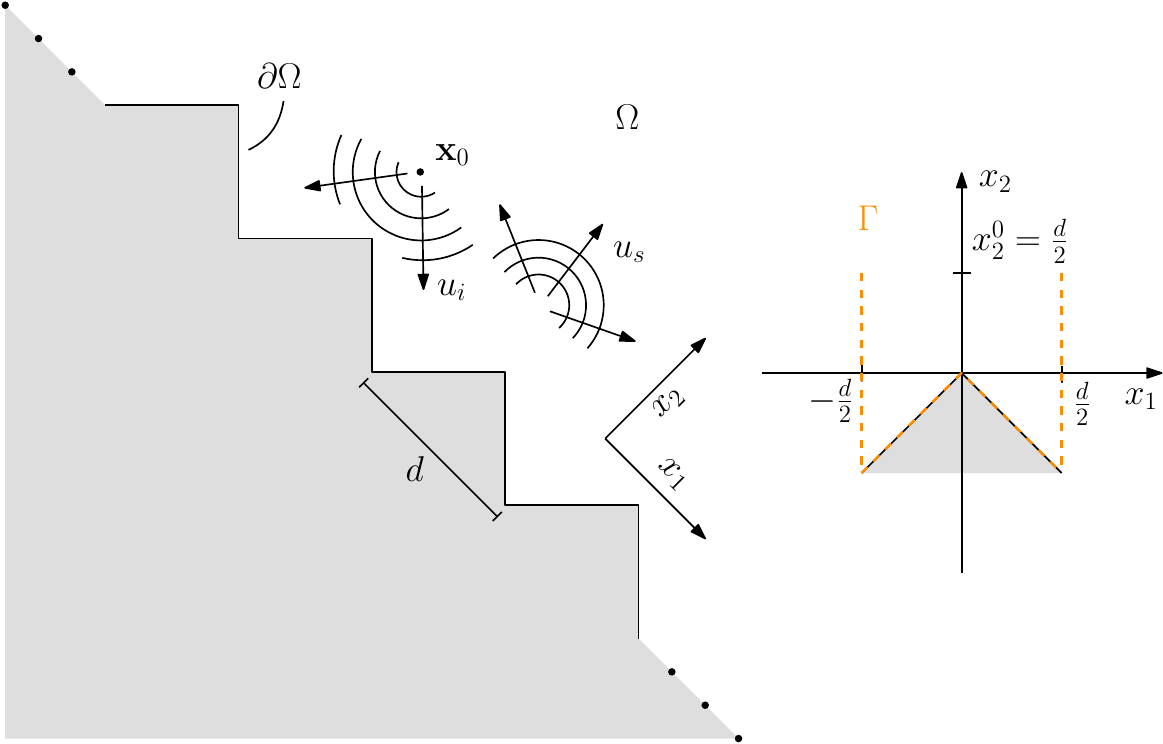}}
    \caption{Left: Photograph of \emph{El Castillo} at Chichen Itza, Mexico. Center: 2D infinite staircase geometry with coordinates used. Right: a single period $\Gamma$ of the boundary, now shown with $x_1$ horizontal (the orientation used throughout). \label{stair-geometry}}
\end{figure}

The remainder of this extended Introduction
sets up definitions needed to convert \cref{helmnp,neubc}
into BVPs posed on a bounded unit cell,
and overviews the rest of the paper.
The solution to \cref{helmnp,neubc}
may be decomposed into a linear combination of
\emph{quasiperiodic} solutions,
via what is known as the array scanning method
in the engineering literature \cite{munk1979plane,rana1981current,capolino2005mode},
or inverse Floquet--Bloch transform in mathematics \cite{lechleiter2017convergent,ruming21}.
This family of quasiperiodic solution is parameterized
by $\kappa\in\R$, the horizontal (along-surface) wavenumber,
and each member of the family has the phased translational symmetry
$u_\kappa(x_1 + nd, x_2) = \alpha^n u_\kappa(x_1, x_2)$, for all $n\in\Z$
and $\xx\in\Omega$,
where the \emph{Bloch phase} is $\alpha = e^{i\kappa d}$.
The set of \emph{plane waves} of the form $e^{i\kk\cdot\xx}$
obeying the homogeneous Helmholtz equation and quasiperiodic symmetry is discrete, namely $\kk_n = (\kappa_n,k_n)$, where
\be
\kappa_n := \kappa + 2\pi n/d, \qquad n \in \Z
\label{kapn}
\ee
is the shifted lattice of horizontal wavenumbers with the same $\alpha$,
and
\be
k_n := +\sqrt{\omega^2 - \kappa_n^2}
\label{kn}
\ee
is the vertical wavenumber. 
$k_n$ is either positive real (upwards-propagating), zero (horizontally propagating), or positive imaginary (upwards-decaying).

Then, for a given $\kappa$, the quasiperiodic solution $u=u_\kappa$ solves the
BVP
\begin{alignat}{3}
  -(\Delta + \omega^2)u &= f_\kappa\quad
  &&\text{in } \Omega, \quad &&\text{\textit{(PDE)}} \label{helm-pde} \\
  u_n &= 0 \quad &&\text{on } \partial \Omega,\quad
  &&\text{\textit{(boundary condition)}} \label{helm-bc} \\
  u(x_1 + d, x_2) &= \alpha u(x_1, x_2) &&(x_1, x_2) \in \Omega, \quad &&\text{\textit{(quasiperiodicity)}} \label{qp-cond} \\
  u(x_1, x_2) &= \sum_{n \in \mathbb{Z}} c_n e^{i(\kappa_nx_1 + k_n x_2)}, \quad && x_2 > x_2^{0}, \quad &&\text{\textit{(radiation condition)} \label{uprc}}
\end{alignat}
where $x_2^{(0)}$ is any height lying above the support of the source and above
$\pO$, and $c_n$ are amplitudes of the outgoing plane waves.
The source $f_\kappa$ must naturally
also be quasiperiodic: it is a quasiperiodized version of the
desired right-hand side $\delta(\xx-\xx_0)$, namely
\be
f_\kappa(\xx) := \sum_{n = -\infty}^{\infty}  e^{in\kappa d} \delta(\xx-\xx_0-n\lv),
\label{asm-setup-per}
\ee
where $\lv:=(d,0)$ is the lattice vector.
Fixing $\omega>0$, the above BVP has a unique solution
for all $\kappa\in\R$ apart from possibly a discrete set
(which will correspond to trapped modes, and discussed
in detail in \cref{trapped-modes}) \cite{bonnetBDS}.
The above also extends to $\kappa\in\C$; see \cite[Sec.~4.1]{shipmanreview}.

Each member $u_\kappa$ of the above quasiperiodic solution family
is the acoustic response of the periodic geometry to
an infinite array of phased point sources.
The key observation (giving ``array scanning'' its name)
is that one may cancel out all but the central
source by integration over $\kappa$ in the first Brillouin zone:
\be
\delta(\xx-\xx_0) = \frac{d}{2\pi} \int_{-\pi/d}^{\pi/d} f_\kappa(\xx) \d\kappa.
\label{asm-setup-single}
\ee
This exploits the fact that $\int_{-\pi/d}^{\pi/d} e^{in\kappa d} \d\kappa = 2\pi/d$ for $n=0$, and zero
for $n\in\Z, n\neq 0$.
Since this is the right-hand side in \cref{helm-pde},
then, by linearity, performing the same integral over the solution
family $u_\kappa$,
\be
u(\xx) = \frac{d}{2\pi} \int_{-\pi/d}^{\pi/d} u_\kappa(\xx) \d\kappa
\label{asmu}
\ee
recovers $u$, the solution to \cref{helm-pde,NBC} with
upward-propagating radiation conditions.
\begin{rmk} 
  \label{trappedrc}
  In the case where $u_\kappa$ is not unique for some discrete
  real $\kappa$ values, the integral \cref{asmu} has poles
  on the real axis corresponding to 
  evanescent waves trapped on the corrugated surface.
  General conditions for existence of such Neumann boundary condition trapped modes are not known \cite[p.11--12]{wilcox},
  although they are well known, and proven to exist for certain geometries
  \cite{evans93,gotlib00}.
  The mathematical formulation of causal
  outgoing radiation conditions in the presence of trapped modes
  is subtle, relying on the \emph{limiting absorption principle}
  (for dielectric cases see \cite{kirsch17,epsteinopen2}).
  We did not find the sound-hard case in the literature,
  but, following Zhang \cite{ruming21}, we
  define this by the topology of the integration contour with respect to the poles.
  This is presented in \cref{singleptsrc}.
\end{rmk}
The integration interval $\kappa \in [-\pi/d, \pi/d]$
is known as the \emph{first Brillouin zone},
and covers the family of solutions.
This is because $u_\kappa$ is the same for all $\kappa$ in the set
\cref{kapn} (since $\alpha$ is the same, and so is the
set of plane waves in the radiation condition, up to relabeling).
The integration may also be thought of as over $\alpha$ on the unit
circle \cite{ruming21}.
We present a high-order accurate
quadrature method for the above integral \cref{asmu}. This is
crucial for efficiency, because each such quadrature node demands a new
BVP solution of \cref{helm-pde,helm-bc,qp-cond,uprc}.
Here,
care is needed due to the possibility of poles mentioned in the above remark,
but also because the
integrand contains two square-root singularities (with associated branch cuts)
at so-called \emph{(Rayleigh--)Wood anomalies}
\cite{wood,fano1941theory,hessel1965new}.
These anomalies are defined as pairs $(\omega,\kappa)$ where
$k_n =0 $ for some $n\in\Z$ in \cref{kn}
(see \cref{wood-anom-bz}),
resulting in a horizontal plane wave.
Although the quasiperiodic BVP remains well-posed,
it poses challenges for Green's function based numerical methods
\cite{arens06,brunohaslam09,barnett_repr_QPS_2D,delourme14,cho2015robust}.
Following Zhang \cite{ruming21}
in the setting of closed waveguides,
we propose contour deformation to complex $\kappa$
to avoid such anomalies and poles.
Rather than the piecewise-smooth contours of that work, we
use a more efficient analytic contour deformation,
and optimize a deformation parameter to minimize the number of nodes needed;
see \cref{singleptsrc}.

Boundary integral equation (BIE) methods
\cite{coltonkress_scatt,kress1989linear} are especially suited for the high-order accurate solution of each quasiperiodic BVP.
Since they operate by first converting the PDE to a boundary integral, then discretizing it to yield a dense linear
system, they only require the evaluation of a 1D integral instead of the
discretization of a (truncated) infinite 2D domain. This reduction of dimensionality significantly reduces
the number of unknowns, and allows for an easy increase in order of accuracy via high-order quadrature rules. 
Furthermore, the staircase geometry has two corners per period,
one of which induces fractional-power-law singularities in $u$;
these are easily handled with BIE.
In contrast, a finite differencing (FD) or finite elements (FEM) would require
meshing of the domain in a manner respecting the corner singularity,
and explicit handling of the radiation condition \cref{uprc}.
Both FD \cite{lechleiter2017convergent} and FEM \cite{zhang2018high}
have been combined with array scanning to solve the scattering from periodic surfaces,
but using only low-order spatial discretization.

There exist other boundary-based approaches.
These include meshfree methods such as the method of
fundamental solutions (MFS) \cite{fairweather_mfs,cheng_mfs_overview,barnett2008stability}, and the plane waves method \cite{alves2005numerical}.
These share common roots with BIE methods in that the solution is constructed
from Helmholtz solutions in the domain,
but generally give ill-conditioned systems.
For scattering by periodic surfaces, a family of methods exists based on the Rayleigh
hypothesis (Rayleigh methods) \cite[p.\ 17]{petit}, \cite{millar1973rayleigh}.
These assume that an expansion of the
scattered field like \cref{uprc} is valid close to and on the surface, and
approximate the solution as a truncation of \cref{uprc}. This assumption, however, does not generally hold.
Fast approximate solutions are most commonly derived using the
Helmholtz--Kirchhoff approximation \cite{meecham1956use}, which assumes that
each point on the boundary scatters as if it was a plane at a slope matching
that of the boundary. This represents a short-wavelength limit and is closely
related to the geometrical acoustics
\cite{keller1958geometrical,keller1962geometrical}; its validity has been
studied thoroughly \cite{richards2018acoustic}. 

To use boundary integral methods, one
splits the physical solution as $u = u_i + u_s$, where
the \emph{incident wave}
$u_i$ solves \cref{helm-pde}
with free-space radiation conditions and no boundary conditions,
and is thus known analytically,
whilst the unknown \emph{scattered wave}
$u_s$ solves the homogeneous version of \cref{helm-pde},
with inhomogeneous boundary condition $(u_s)_n = -(u_i)_n$ on $\pO$,
and \cref{qp-cond,uprc}.
A BIE is then used to solve this BVP for $u_s$, as presented in \cref{bie}.

A BIE formulation on the infinite surface $\pO$ would not be numerically
feasible, and its truncated solution converges very slowly
since the contributions of distant sources decay only
like $1/\sqrt{r}$.
However, the computation of the quasiperiodic solution may 
be reduced to a single unit cell of the boundary. This periodization process involves replacing in the integral kernels the
free-space Green's function $\Phi(\xx)$, defined as the radiative solution to
\be
-(\Delta + \omega^2)\Phi(\xx) = \delta(\xx),
\ee
with the quasiperiodic
Green's function $\phip(\xx)$ defined as the upwards- and downwards-radiating
solution to
\be
-(\Delta + \omega^2)\phip(\xx) = 
\sum_{n=-\infty}^{\infty} \alpha^n \delta(x_1 - nd) \delta(x_2).
\label{phip}
\ee
The periodic Green's function is therefore an infinite phased sum of single point-source Green's functions.
The sum
is slowly convergent and cannot be used
directly \cite{linton98}.
Yet a wide range of methods exist
for the rapidly-convergent approximation of $\phip$,
including reformulation in terms of quickly convergent lattice sums
\cite{linton2010lattice}.
In \cref{periodization}, we discuss one such efficient computation of $\phip$
based on \cite{yasumoto1999efficient}.
By the use of contour deformation to complex $\kappa$,
we will avoid Wood anomalies where \cref{phip} does not exist.

With the periodic Green's function in
hand, \cref{discretization} describes the boundary layer representation of the
solution and the discretization of the boundary, including the treatment of the singularities present in
the boundary integral at the corners of the boundary.
In general, these may be handled analytically,
\eg via Gauss--Jacobi quadrature \cite{tsalamengas2016gauss},
or conformal mapping \cite{driscoll2002schwarz}; or
handled to high order by generalized Gauss
quadrature \cite{bremer2010universal}, recursive compressed inverse
preconditioning \cite{helsing2008corner}, or rational function approximation
\cite{gopal2019new}.
In \cref{beyond-uc} that follows, we describe how the total field is
reconstructed outside the unit cell, and verify the high-order accuracy of the
solution with convergence tests which exploit flux conservation.

\begin{rmk}
  A distributed spatial source (right-hand side in \cref{helm-pde})
  may also be handled by a slight
  generalization of our framework:
  one numerically computes their incident wave $u_i$ by
  convolution of the source function with $\phip$. The
  BIE solutions for $u_s$ then proceed as before.
\end{rmk}

In \cref{trapped-modes} we will show how numerically to
locate eigenparameters $(\omega,\kappa)$ for which trapped modes
exist.
Such modes are eigenfunctions, i.e., nontrivial homogeneous solutions to
the BVP \cref{helm-pde,helm-bc,qp-cond,uprc}.
As \cref{trappedrc} implied, at each $\omega$, such $\kappa$ values
are vital to know since they induce poles in the array scanning
(inverse Floquet--Bloch) integral.
Yet, the mode \emph{dispersion relation}---the dependence of $\omega$
vs $\kappa$ for trapped modes---also will provide a model to
understand the ``raindrop effect'' (chirp-like response)
for the time-domain problem \cref{WE,WEBC}.
For this reason,
we present a strategy to solve for the trapped $\omega$ at a given $\kappa$.
This involves rootfinding on
the Fredholm determinant of the $\omega$-dependent integral operator, as done in \cite{zhaodet}.
A trapped mode may be reconstructed from the eigenfunction
of the Fredholm 2nd-kind integral equation.
From the dispersion relation, we
derive their group velocities; the speed at which a given trapped mode
propagates along the surface.
We use an approximate ray model (neglecting amplitudes) to predict 
the arrival times of different frequencies at the bottom of a staircase
modeled after the \emph{El Castillo} pyramid,
in order to understand the chirp-like sounds observed.

We also pose, and answer in Section~\ref{power-in-trapped}, the following:
how can one efficiently use the array scanning method to report
the power carried away by (left- or right-going) trapped modes,
as opposed to power radiated upwards away from the surface?
Characterizing this division of radiated power as a function of frequency
is crucial in related engineering applications such as
point-source radiators in nanophotonics and acoustic metamaterials.
Finally, we discuss avenues for future work in \cref{future}.

\section{BIE formulation, periodization, and discretization}
\label{bie}

Here we present the numerical method for solving the quasiperiodic
BVP \cref{helm-pde,helm-bc,qp-cond,uprc}, at a given $\omega>0$ and $\kappa\in\C$.
The physical wave (potential) is written $u = u_i + u_s$.
Fixing the source $\xx_0\in\R^2$,
we take the incident wave to be the quasiperiodic function
\be
u_i(\xx) = \phip(\xx,\xx_0), \quad \xx \in \R^2,
\ee
where we use the notation $\phip(\xx,\yy) = \phip(\xx-\yy)$, recalling
\cref{phip}.
The PDE and boundary condition for $u$ \cref{helm-pde,helm-bc} imply that
$u_s$ solves the BVP
\begin{alignat}{2}
    (\Delta + \omega^2)u_s &= 0 \quad &&\text{in } \Omega, \label{helm-us} \\
  (u_s)_n &= - (u_i)_n \quad &&\text{on } \partial\Omega,
  \label{NBC}
\end{alignat}
with $u_s$ obeying the quasiperiodicity and radiation conditions \cref{qp-cond,uprc}.
Since the PDE is now homogeneous, a BIE solution becomes possible.

The unknown scattered field $u_s$ is
represented by
a quasiperiodic single-layer potential
\be
u_s(\xx) = (\S \sigma)(\xx) := \int_\Gamma \phip(\xx, \yy) \sigma(\yy) \d s_\yy, \quad \xx \in \R^2,
\label{rep}
\ee
where the usual fundamental solution $\Phi(\xx, \yy)$ has been replaced with
its quasiperiodic counterpart,
and $\Gamma := \partial\Omega \cap \{|x_1|\le d/2\}$ is a single unit cell (period) of the boundary.
As $\xx$ approaches the boundary from either side, let us define 
\begin{align}
    u_s^{\pm}(\xx) &:= \lim_{h \to 0^{\pm}} u_s(\xx + h\vec{n}_\xx)
    \label{upm}, \quad \xx \in \Gamma, \\
    (u_s)_n^{\pm}(\xx) &:= \lim_{h \to 0^{\pm}} \vec{n}_\xx \cdot \nabla u_s(\xx + h\vec{n}_\xx), \quad \xx \in \Gamma.
  \label{unpm}
\end{align}
Here $\vec{n}_\xx$ is the unit normal to the boundary at the target point $\xx$.
The above representation then satisfies the jump relations (see \cite[Ch 6.3]{kress1989linear})
\begin{align}
    u_s^{\pm} &=  S \sigma, \quad \mbox{ on }\;\Gamma,\\      
    (u_s)_n^{\pm} &= \left( D^{\mathrm{T}} \mp \frac{1}{2} I \right) \sigma, \quad \mbox{ on }\;\Gamma,
\end{align}
where $D^T$ denotes the adjoint double-layer operator on $\Gamma$,
namely the operator with kernel
$\partial \phip(\xx, \yy)/\partial \vec{n}_\xx$ taken in the principal value sense.

The single-layer representation for $u_s$ automatically satisfies \cref{helm-us}, and
the boundary condition can be written in terms of the single-layer jump
condition as
\be
(I - 2D^{\mathrm{T}})\sigma = 2 (u_i)_{n} \quad \text{ on }\;\Gamma. \label{fredholm-ii-cts}
\ee
This is a Fredholm integral equation of the second kind.
\begin{rmk}
  The expert reader may wonder why a combined-field representation
  (``CFIE'') is not
  needed here to prevent spurious resonances,
  as in the case of a bounded obstacle
  \cite{coltonkress}.
  In fact \cref{rep} is sufficient when the unbounded
  $\pO$ is a graph of a function, because then the Dirichlet
  BVP in the complementary domain $\R^2 \backslash \overline\Omega$
  is unique for any $\omega$; see \cref{SLP} below and its proof.
\end{rmk}

\begin{figure}  
\centering
    \includegraphics[width = 0.5\textwidth]{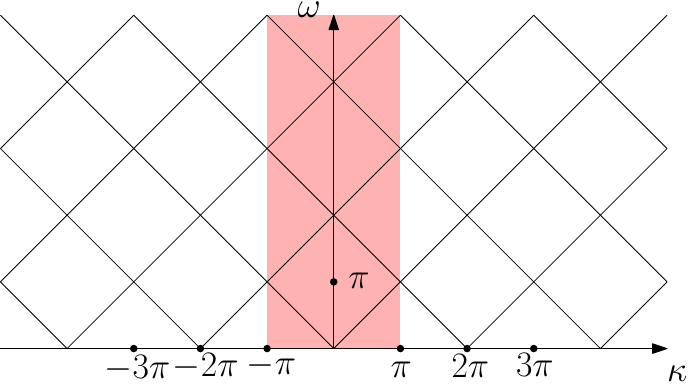}
    \caption{Location of Wood anomalies (black lines) and the first Brillouin zone (red shaded region of the $\kappa$ axis) in the $\omega$-$\kappa$ plane,
      for the case of spatial periodicity $d=1$.  
    \label{wood-anom-bz}}
\end{figure}

\subsection{Evaluation of quasiperiodic Green's functions \label{periodization}}

Here we describe an efficient numerical scheme to
evaluate $\phip(\xx)$ appearing in the above BIE,
using a local expansion about the origin,
with lattice sum coefficients evaluated following
\cite{yasumoto1999efficient}.
Note that, since it is only valid up to a vertical height comparable to the unit
cell width, a different plane-wave representation
given in \cref{beyond-uc} will be used for $x_2$ values beyond this.
Recall the definitions \cref{kapn,kn}, and that
a Wood anomaly is a point in the $(\omega,\kappa)$ parameter
plane where $\kappa_n$ crosses $\kappa_n = \omega$.
At such anomalies $\phip$
does not exist (e.g. see \cite{linton98,barnett_repr_QPS_2D}).
We will thus assume that $(\omega,\kappa)$ is not at a Wood anomaly
(and direct the reader to \cite{barnett_repr_QPS_2D,delourme14} for
methods for quasiperiodic problems precisely at Wood anomalies).

The quasiperiodic Helmholtz Green's function in 2D is given by \cite{linton98}
\be 
\phip(x_1, x_2; \omega, d, \alpha_0) = \frac{i}{4}\sum_{n = -\infty}^{\infty} H_0^{(1)} \left( \omega\sqrt{(x_1 - nd)^2 + x_2^2 }\right) \alpha^n.
\label{phip-def}
\ee
The sum's slow, conditional
convergence means it is of little practical use. We follow the approach
described in \cite{yasumoto1999efficient} to rewrite it in terms of
rapidly
convergent
lattice sums. Our derivation, however, differs in two key ways; we therefore
reproduce some of the calculation below for convenience.
Graf's
addition theorem \cite[(10.23.7)]{dlmf}
is used first to expand each of the
$H_0^{(1)}$ centered on source point with an $x_1$ coordinate outside of the
unit cell, $|x_1| > \tfrac{d}{2}$, around the equivalent point inside the unit
cell. The resulting expression contains the Bessel functions $H_n^{(1)}$ and
$J_n$. Terms multiplying the same-order $J_n$ are collected; their
coefficient, the $n$th order lattice sum $S_n$, is expressed as a contour integral.
For faster convergence, following \cite{barnett_repr_QPS_2D},
we instead exclude $|n| \leq 1$ from the sum in
\cref{phip-def} and add them together directly; this can be thought of
splitting the periodic Green's function into a near and far component:
\be
    \phip = \phipn + \phipf,
\ee
with
\be
    \phipn = \frac{i}{4}\left[ H_0^{(1)}\left(\omega\rho_0\right)  
                             + \alpha^{-1} H_0^{(1)}\left(\omega\rho_1\right)\
                             + \alpha H_0^{(1)}\left(\omega\rho_{-1}\right)\right],
\ee
where $\rho_j = \sqrt{(x_1 - jd)^2 + x_2^2}$. We use Graf's theorem on the remaining terms to write 
\be
    \phipf = \frac{i}{4}S_0(\omega d, \kappa)J_0(\omega \rho_0) + \frac{i}{2}\sum_{n = 1}^{\infty} S_n(\omega d, \kappa) J_n(\omega\rho_0)\cos(n\phi), \label{phipf}
\ee
where $\phi$ is the angle describing the source-target displacement vector
$(x_1, x_2)$ in polar coordinates.
The lattice sums $S_n$ are independent of the source-target displacement,
therefore they only need to be
computed once for each value of $\omega$ and $\kappa$. We follow
\cite{yasumoto1999efficient} to derive the integral representation 
\begin{align}
    S_n(\omega d, \kappa) \approx  -ie^{i\tfrac{\pi}{4}}\frac{\sqrt{2}}{\pi} \bigg[ &(-1)^n \alpha^{-2} \int_0^a [G_n(\tau) + G_n(-\tau)] F(\tau; \omega d, \kappa)\d t \nonumber\\
    &+ \alpha^2 \int_0^a [G_n(\tau) + G_n(-\tau)] F(\tau; \omega d, -\kappa)\d t\bigg], \label{sn-int}
\end{align}
with $\tau = (1 - i)t$ and
\begin{align}
    G_n(t) &= \left( t - i\sqrt{1 - t^2}\right)^n, \\
    F(t; \omega d, \kappa) &= \frac{e^{2i\omega d\sqrt{1 - t^2}}}{\sqrt{1 - t^2}[1 - e^{i\omega d(\sqrt{1 - t^2} - \kappa / (\omega d))} ]}.
\end{align}
The integrals appearing in $S_n$ may be evaluated numerically using the trapezoidal rule.
The reason behind \cref{sn-int} being an approximate expression instead of an
equality is that upper limit of the integrals, $a$, has been truncated from
$\infty$, exploiting the fact that the integrand decays quickly for $t \gg 1$
due to the exponential factor in $F$. 

Three convergence parameters are needed here: the number of
terms $N$ at which the sum in $\phipf$ (in \cref{phipf}) is truncated, the
upper limit $a$ of the integrals in \cref{sn-int},
and the number of nodes $M$ used
in the trapezoidal rule to calculate \cref{sn-int}.
For the typical $\omega$ and accuracies that we present, we found $N = 40$, $a = 15$, and $M = 10^4$ sufficient.

\subsection{Discretizing the boundary integral \label{discretization}}

We discretize and solve the integral equation \cref{fredholm-ii-cts} using the
Nystr\"om method \cite[Sec 12.2]{kress1989linear}, 
summarized below. First, \cref{fredholm-ii-cts} is written in the standard form
\be
    \sigma(t) - \int_a^b K(t, s)\sigma(s)\d s = f(t) \quad \text{for} \quad t \in [a, b],
\ee
where $K(t, s)$ is the appropriate kernel function, and the boundary is
parameterized by $s$ that runs from $s = a$ to $s = b$.
Then, the unknown \emph{density} function $\sigma(t)$ is approximated by another function $\sigma_n(t)$ that obeys
\be
    \sigma_n(t) - \sum_{j = 1}^n w_j  K(t, s_j) \sigma_n(s_j) = f(t), \label{nystrom-disc}
\ee
where the integral has been replaced with a quadrature formula with nodes
$\{s_i\}_{i = 1}^n$ and weights $\{w_i\}_{i=1}^{n}$, and the kernel function $K$ with a rank-$n$ operator. Then
the values of $\sigma_n(t)$ at the nodes $\{s_i\}$, $\{\sigma_i^{(n)}\} :=
\{\sigma_n(s_i)\}_{i = 1}^{n}$ satisfy the linear system
\be
    \sigma_i^{(n)} - \sum_{j = 1}^n w_j K(s_i, s_j)\sigma_j^{(n)} = f(s_i), \quad \forall i = 1, 2, \ldots, n,
    \label{nystrom-sys}
\ee
or in matrix notation,
\be 
    A\vec{\sigma}^{(n)} = \vec{f}, \label{nystrom-linsys-vecnot}
\ee
with $A_{ij} = \delta_{ij} - K(s_i, s_j)w_j$. 
Then if any vector 
$\vec{\sigma}^{(n)} := \{\sigma_i^{(n)}\}$ 
solves the above system, then $\sigma^{(n)}(t)$ at any $t$ can be
reconstructed as 
\be
    \sigma^{(n)}(t) = f(t) + \sum_{j = 1}^n w_j K(t, s_i)\sigma_j^{(n)}. \label{nystrom-rec}
\ee
Substituting the kernel associated with our exterior Helmholtz problem with
boundary $\partial \Omega := z(s)$, the $A$ matrix in \cref{nystrom-linsys-vecnot} becomes:
\be
A_{ij} = \delta_{ij} - 2 w_j z'(s_j) \frac{\partial\phip(s_i, s_j)}{\partial \vec{n}_{s_i}}, \quad i,j = 1,2,\dots, n, \label{A-expr}
\ee
while the components of $\vec{f}$ are given by
\be 
f_j = 2(u_i)_n(s_j),  \quad j = 1, 2, \dots, n. \label{f-expr}
\ee
Finally, the gradient of $\phip$, necessary for filling the $A$-matrix in \cref{A-expr}, is
\begin{align}
    \nabla \phip(\vec{x}, \vec{y}) = &\frac{i}{4}\bigg(- \omega H_1^{(1)}(\omega \rho_0) \frac{\vec{x} - \vec{y}}{\rho_0} 
                     - \omega H_1^{(1)}(\omega \rho_1) \frac{\vec{x} - \vec{y} - d\vec{e}_{x_1}}{\rho_1} \nonumber\\
                     &- \omega H_1^{(1)}(\omega \rho_{-1}) \frac{\vec{x} - \vec{y} + d\vec{e}_{x_1}}{\rho_{-1}}
                     - \omega J_1(\omega \rho_0) S_0 \frac{\vec{x} - \vec{y}}{\rho_0} \label{grad-phip}\\
                     &+ 2 \sum_{n = 1}^{\infty} \omega J_n'(\omega \rho_0)\cos(n\phi)S_n \frac{\vec{x} - \vec{y}}{\rho_0}  + nJ_n(\omega \rho_0) \sin(n\phi) S_n\frac{\hat{\vec{n}}_{\vec{x} - \vec{y}}}{\rho_0}\bigg), \nonumber
\end{align}
where $\rho_j := \sqrt{(x_1 - y_1 - jd)^2 + (x_2 - y_2)^2}$, $\hat{\vec{n}}_{\vec{x} - \vec{y}} := (x_2 - y_2, -(x_1 - y_1))/\rho_0$ is
a unit normal to the displacement vector $\vec{x} - \vec{y}$, and $\vec{e}_{x_1}$ is a unit vector in the $x_1$
direction.

The key task for achieving high-order accuracy is in choosing the appropriate quadrature nodes and weights appearing in \cref{nystrom-disc}. 
For smooth boundaries, efficient quadrature rules are known,
based on either the global periodic trapezoid rule, or high-order
panel-wise quadrature. 
For boundaries with corners, however, the density function $\sigma$ is singular
at the corners. The closeness of this singularity limits the radius of the Bernstein ellipse associated with the interpolant used in Gaussian quadrature rules, which in turn degrades the accuracy of
interpolation and quadrature \cite[Ch 8, 19]{atap}. To maintain the order of
accuracy whilst keeping the number of quadrature nodes the same, the size of
the quadrature interval needs to be reduced as the corner is approached, \ie
the quadrature grid needs to be refined according to the features of the
boundary \cite{greengard2014fast}.
  We divide the two sides of unit cell boundary into equally-spaced
panels, and use $P$-th order Gauss--Legendre nodes and weights on each, \ie
$s_i$ and $w_i$ are always the Gauss--Legendre nodes and weights on
the standard interval $[-1, 1]$. Then
for each panel with a corner as an endpoint, we divide the panel in a
$1:(r-1)$ ratio with $r \geq 2$, using the same $P$-th order Gauss--Legendre
scheme on each. With $r = 2$ (dyadic refinement), the net error after
$\mathcal{O}(\log_2(1/\varepsilon))$ levels of refinement will be
$\mathcal{O}\bigl(e^{-P}\log_2(1/\varepsilon)\bigr)$, where $\varepsilon$ is a given
numerical precision \cite{greengard2014fast}.
Moving from a single set of $n$ global quadrature nodes to panel quadrature with $Q$ panels, each with $P$ nodes, modifies the expression \cref{nystrom-sys} to
\be
\sigma_{j, i} - \sum_{q = 1}^{Q} \sum_{p = 1}^{P} w_{j, i, q, p} K(s_{j, i}, s_{q, p}) \sigma_{q, p} = f_{j, i}, \quad  j = 1, \dots, Q, \; i = 1, \dots, P,
\ee
where $s_{j, i}$ is the $i$-th Gauss--Legendre node on the $j$-th panel,
$\sigma_{j, i}$, $f_{j, i}$, $w_{j, i}$ are the corresponding density, boundary
data, and quadrature weights, and $K(s_{j, i}, s_{q, p})$ is the value of the
kernel associated with the given pair of quadrature nodes. 

Typically, different quadrature rules are needed depending on where $s_{j, i}$
and $s_{q, p}$ (the ``target'' and ``source'' points) lie relative to each
other \cite{bremer2010universal,atkinson1997numerical}, with special care
required if they lie on adjacent or the same panel.
The need for this can be seen by inspecting \cref{grad-phip} as $\vec{x} \to
\vec{y}$: simple Gauss--Legendre rules cannot capture the (weak) singularity
that emerges in this limit, therefore the close-to-diagonal entries in $A$ will
not be accurate. In the special case of the staircase, however, it is possible
to achieve higher order accuracy without invoking special quadrature rules:
since quadrature panels are straight lines, the interaction between source and target
points on the same or colinear panels vanishes. This is
clear from \cref{A-expr} and \cref{grad-phip}: $\vec{x} - \vec{y}$ is
perpendicular to $\vec{n}_\xx$ in this case.

Decrease in accuracy can arise from
catastrophic cancellation if $|\vec{x} - \vec{y}| \ll |\vec{x}|, |\vec{y}|$,
\eg nodes lying on small panels either side of a
corner. To mitigate this, we parameterize the quadrature panels such that the
quadrature nodes' positions are measured relative to the nearest corner.

Note that reconstructing the solution $u_s$ close to the boundary
would require special quadrature rules, since the position
vectors of the target and source points are not necessarily colinear. Since such special rules are beyond the scope of this paper, the numerical solution will not be accurate
within roughly the length of the nearest boundary panel.
Fortunately this will not prevent us from extracting the trapped power
to full accuracy.

\begin{figure}[tb]
\centering
    \includegraphics[width = 0.7\textwidth]{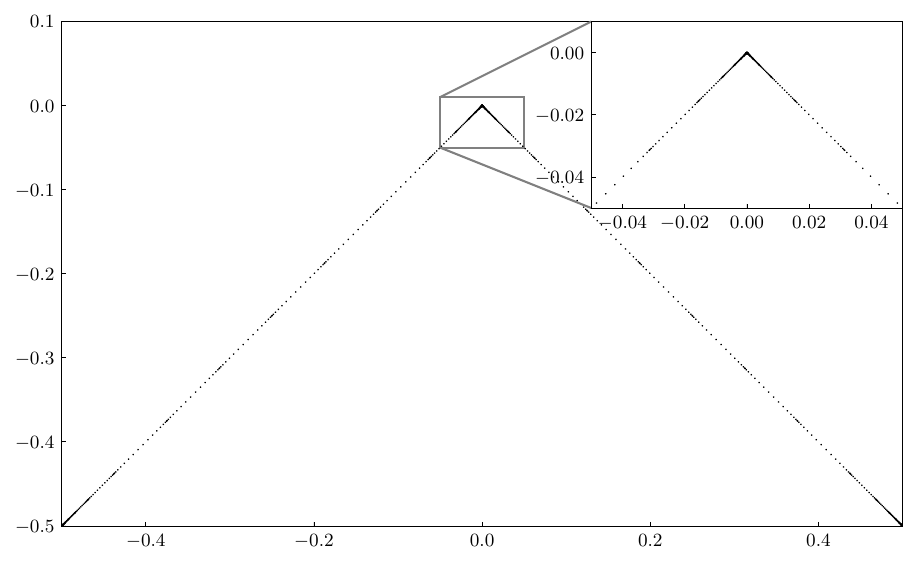}
    \caption{Discretization nodes on
      a single unit cell $\Gamma$ of the boundary.
      The underlying coarse discretization has $8$ equal panels on each
      straight line.
      Panels touching corners have then been subdivided dyadically
      $10$ times, with shrinkage ratio $r = 2$.
      Each resulting panel was populated with $16$ Gauss--Legendre quadrature
      nodes. An inset shows the result of the refinement.
    \label{stair-disc}}
\end{figure}

\subsection{Reconstructing the solution \label{beyond-uc}}

Once the density evaluated at the boundary nodes, $\vec{\sigma}^{(n)}$, is
obtained, on may compute the scattered field $u_s$ at a point $\xx$ inside the unit cell from its single-layer representation,
\be
u_s(\xx) = \sum_{j = 1}^n w_j z'(s_j) \phip(\xx, s_j)\sigma^{(n)}_j. 
\ee
(After which, $u(t) = u_i(t) + u_s(t)$.) The lattice sum representation of
$\phip$, and hence the above expression, quickly loses accuracy outside of the
circle $(x_1 + x_2)^2 = \tfrac{3d}{2}$ due to the application of Graf's
addition theorem\footnote{This would be a smaller, radius-$\tfrac{d}{2}$ circle
were $n = \pm 1$ not excluded from \cref{phipf}.} .

 Above the unit cell, the upwards propagating radiation condition
may be exploited as follows. Let $x_j:= -\tfrac{d}{2} + \tfrac{jd}{N}$ for $j =
0, 1, \ldots, N-1$ and $u^0_j := u(x_j, x_2^0)$. Evaluating the radiation condition at $x_2 = x_2^0$, notice that
after rearranging, it takes the form of a discrete Fourier transform: 
\begin{gather}
    u^0_j e^{-i\kappa x_j} = \sum_{n \in \mathbb{Z}} \tilde{c}_n e^{i 2\pi n x_j},
\end{gather}
with $\tilde{c}_n = c_n e^{i k_n x_2^0}$. Therefore, using the convention from \cite[Ch.\ 12.1]{numerical_rec},
\be
\tilde{c}_n = \frac{1}{N}\mathcal{F}\left(\{u^0_j e^{-i\kappa x_j}\}_{i=0}^{N-1}\right)_n,
\ee
and 
\be
u(x_j, x_2) = N \cdot \mathcal{F}^{-1}\left( \{\tilde{c}_n e^{i(k_n(x_2 - x_2^0) + i\kappa x_j  )}\}_{n=0}^{N-1} \right)_j.
\ee
In the $n$th neighboring unit cell, \ie $ \tfrac{(2n - 1)d}{2} <
x_1 \leq \tfrac{(2n+1)d}{2} $ (at any given $x_2$), the solution is found using the quasiperiodicity
condition,
\be
u(x_1 + nd, x_2) = \alpha^n u(x_1, x_2).
\ee

\subsection{Convergence check via flux conservation \label{conv-check-per}}

In the absence of sources of sinks inside a closed boundary $\Gamma$, the net
acoustic power (flux) leaving $\Gamma$ is zero. How close it is to zero numerically can
be used to measure the accuracy of the method and investigate
convergence in terms of the size of the linear system
\cref{nystrom-linsys-vecnot}. The acoustic power passing through a surface is\footnote{See \cite{shipmanreview},  or follow the derivation in \cite{kaltenbacher2018computational}: start from (95), then substitute in (92), the 3D generalization of (82), and use harmonicity.}
\be
F_{\Gamma} = \Im\left(\int_{\Gamma} \bar{u}u_n\d s\right).
\label{net-flux-formula}
\ee
To test the convergence of the method above, we choose the incident wave to be
a plane wave, $u_i = e^{i\omega\vec{n}_i\cdot \vec{x}}$, where $\vec{n}_i$ is a
unit vector describing its direction of travel.
This corresponds to the limit of moving a quasiperiodic point source array
to infinity in the vertical direction.
Specifically,
$\vec{n}_i = (\cos(\phi_i),
\sin(\phi_i))$ with $-\pi \leq \phi_i \leq 0$. Then $\kappa = \omega
\cos(\phi_i)$.
Let $\Gamma$
be the perimeter of the unit cell as shown in \cref{stair-geometry}, bounded from below by the
stair boundary, above by the line $x_2 = \tfrac{d}{2}$, and
from the two sides by $x_1 = \pm\tfrac{d}{2}$. The net flux through the two
sides has to be zero by symmetry, and it is also zero across the lower edge of
$\Gamma$ due to Neumann boundary conditions. Therefore we have
\be
F_{\Gamma} = \Im \left( \int_{-d/2}^{d/2}\bar{u}\frac{\partial u}{\partial x_2}\d x_1 \right) = 0.
\ee
In \cref{flux-conv-perisrc-neu} we evaluate $F_{\Gamma}$ as we refine the
corner-adjacent panels with $r = 3$. Initially there are $Q = 8$ panels on each
side with $P = 8$ nodes on each, and after $R$ levels of refinement, there are
$Q = 2(8 + 2R)$ panels and a total of $PQ = 128 + 32R$ nodes. The experiment
used $\omega = 4.0$, and $\phi_i = -0.5$.
The figure confirms exponential convergence and suggests that around $10$ levels of
refinement are needed for $\approx 7$ digits of accuracy, and around $22$ for
$10$ digits, at which point the system size is still smaller than $10^3 \times
10^3$. The reduction of error stops at this point, which is roughly consistent
with the spacing between the closest quadrature nodes approaching machine
precision.

\begin{figure}[tbp]
\centering
    \includegraphics[width = \textwidth]{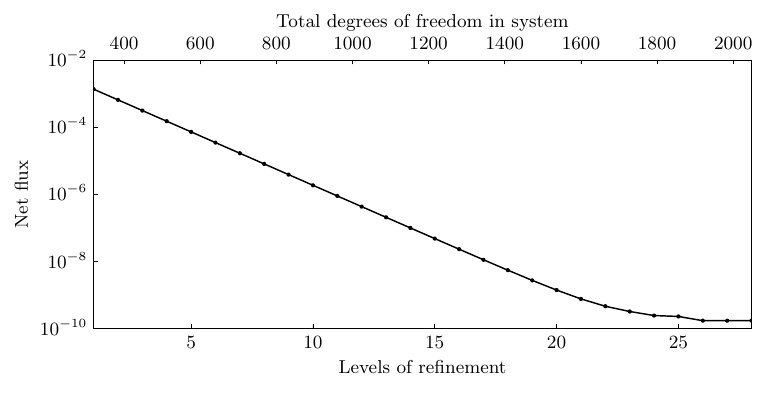}
    \caption{Convergence test via flux conservation for the case of an
    incident plane wave. The figure shows the net flux exiting the central unit
    cell, which is analytically zero, as a function of both the number of times
    the corner-adjacent quadrature panel has been subdivided on the boundary,
    and the size of the resulting linear system. During refinement the panels
    have been split in a $1:2$ ratio ($r = 3$). \label{flux-conv-perisrc-neu}}
\end{figure}

\begin{figure}[th]  
\centering
    \subfloat{\includegraphics[width = 0.5\textwidth]{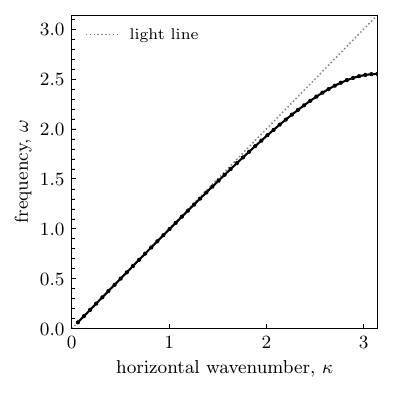}}
    \hfill
    \subfloat{\includegraphics[width = 0.5\textwidth]{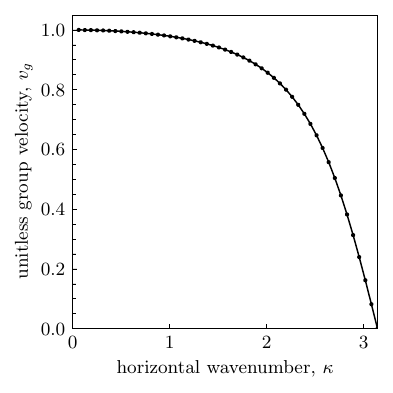}}
    \caption{Left: numerically computed band structure (dispersion relation) for evanescent trapped modes of the $\pi/4$-slope staircase with period $d=1$.
      Only the right (positive) half of the Brillouin zone is shown.
      Dots show the band structure $\omega_\trp(\kappa)$, and the line shows the values $\omega=\kappa$.
      Radiation into the upper half plane is only possible when $\omega>|\kappa|$.
      Right: group velocity $v_g := d\omega_\trp(\kappa)/d\kappa$ plotted
      over the same domain.
      \label{disper}}
\end{figure}

\begin{figure}[th]  
\centering
    \includegraphics[width = \textwidth]{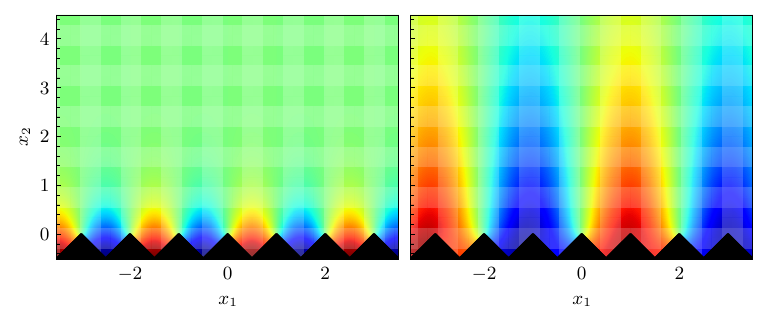}
    \caption{
      The real part of two example trapped modes $\phi$, for
      the $\pi/4$ staircase with $d=1$.
      Left: the highest frequency mode, at $\kappa=\pi$.
      Right: An intermediate frequency mode, at $\kappa = 1.54$.
      \label{modes}}
\end{figure}

\begin{figure}[th]  
\centering
    \includegraphics[width = \textwidth]{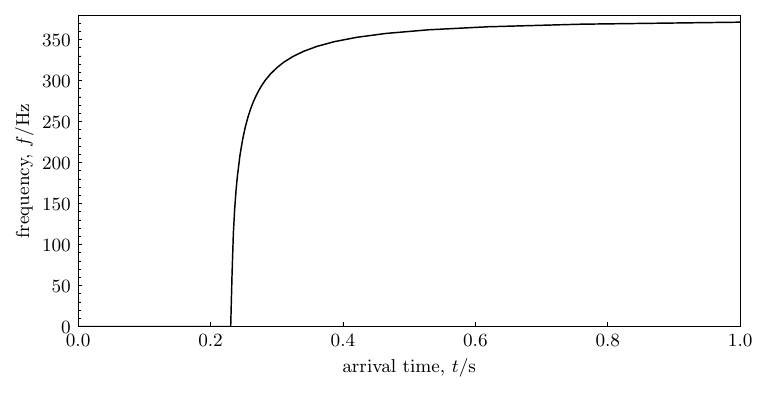}
    \caption{Graph of frequency (in Hz) vs time (in seconds since emission)
      observed a distance $D$ along a sound-hard staircase
      from an impulsive source for the wave equation.
      Parameters for the \emph{El Castillo} staircase are used,
      with $D$ the full 91 steps comprising one side of the pyramid.\label{chirp}}
\end{figure}

\section{Trapped acoustic modes and the raindrop effect \label{trapped-modes}}

In this section we describe and test a Fredholm determinant method to find
evanescent modes trapped by the corrugated sound-hard interface.
Such modes are eigenfunctions, i.e.,
nontrivial solutions to the homogeneous
quasiperiodic BVP \cref{helm-pde,helm-bc,qp-cond,uprc},
for some $\kappa\in\R$ and $\omega>0$.
Their parameters form a continuous families, thus curves in the
$(\omega,\kappa)$ plane, known as the \emph{band structure}
\cite{jobook,shipmanreview}.
Each curve may be described by 
its trapped mode frequency $\omega_\trp(\kappa)$,
commonly referred to as a dispersion relation.
For staircase geometries we find empirically that there is only a single
such trapped frequency at each wavenumber
$0 < |\kappa| \le d/2$ in the Brillouin zone, 
as shown in \cref{disper}.

Recall that the quasiperiodic BVP has the
BIE formulation \cref{fredholm-ii-cts}.
With homogeneous boundary data this becomes the BIE
\be
(I - 2D^T)\sigma = 0,
\label{homogBIE}
\ee
with $D^T$ the adjoint double-layer operator from the previous section.
One might then hope that the condition for a trapped mode to exist
is equivalent to the existence of a nontrivial density $\sigma$ solving
\cref{homogBIE}. We now show that this is indeed so,
thus that the proposed numerical method is robust
(free of spurious resonances).
\begin{thm}\label{SLP} 
  Fix $\omega>0$ and $\kappa\in\R$.
  There is a trapped mode (i.e., a nontrivial $\phi$ solving the
  homogeneous quasiperiodic BVP \cref{helm-pde,helm-bc,qp-cond,uprc})
  if and only if $\Dim \Nul (I - 2D^T) > 0$.
\end{thm}
\begin{proof}
  Let $\sigma$ be a nontrivial solution to $(D^T - I/2)\sigma=0$,
  and let $\phi=\S\sigma$ throughout $\R^2$, recalling that $\S$
  is the quasiperiodic single-layer potential on $\Gamma$ (the
  part of $\pO$ in the central unit cell).
  Then $(\Delta + \omega^2)\phi = 0$ in $\R^2\backslash \pO$.
  By the jump relations, $\phi_n^+ = 0$ (using the notations
  \cref{upm} and \cref{unpm}), and by construction $\phi$
  also satisfies \cref{qp-cond} and \cref{uprc}.
  It remains to show that $\phi$ is nontrivial. Assume this were
  so, then $\phi^+=0$, so by the jump relations, $\phi^-=0$.
  However, $\phi$ in the half-space below $\pO$ would then be
  a homogeneous solution to the (downward-facing)
  quasiperiodic Dirichlet problem,
  which is unique \cite[p.~56]{petit} \cite[Thm.~2.1]{kirsch94}.
  Thus $\phi$ would vanish below $\pO$, so $\phi_n^-=0$. By the
  jump relation $\sigma=\phi_n^--\phi_n^+$ would vanish, a
  contradiction with the hypothesis. Thus $\phi$ is nontrivial,
  hence a trapped mode.

  For the converse, let $\phi$ be a trapped mode,
  then by the quasiperiodic version of Green's representation theorem,
  in the upper domain $\phi = \S \phi_n^+ + \D \phi^+$,
  where $\D$ is the quasiperiodic double-layer potential defined by
  \[  
      (\D \sigma)(\xx) := \int_\Gamma \frac{\partial \phip(\xx, \yy)}{\partial \vec{n}_\yy} \sigma(\yy) \d s_\yy~.
  \]
  Since $\phi_n^+=0$, and taking $\xx$ to $\pO$ from above,
  $\phi^+ = (D+I/2)\phi^+$, showing that $D-I/2$ has a nontrivial
  null vector. By the Fredholm--Riesz theory, the same holds for $D^T-I/2$,
  since it is the adjoint with respect to the bilinear form
  $\langle \psi,\phi \rangle := \int_\Gamma \psi \phi \d s$
  \cite{coltonkress_scatt}.
\end{proof}
This informs our numerical approach: we fix $\kappa$
and solve a nonlinear eigenvalue problem with respect to $\omega$.
We adapt the Fredholm determinant method
of Zhao and the 2nd author \cite{zhaodet}.
$g(\omega) := \det (I - 2D^T)$ is approximated by an $N\times N$
determinant, where $I-2D^T$ is replaced by its 
Nystr\"om matrix with entries given by \cref{A-expr}.
At each $\kappa$, then $\omega_\trp(\kappa)$ is found as a root of
$g(\omega)$, using a simple Newton iteration
with a convergence criterion close to machine accuracy.
For staircase geometries, we find only one such root,
and always with $\omega_\trp(\kappa)<|\kappa|$,
implying that the mode is trapped (frequency is below the light line),
rather than embedded in the continuous spectrum
\cite{jobook,shipmanreview}.
Because the trapped frequencies never intersect the light lines,
the issue of Wood anomalies does not cause a problem in the mode-finding
task.

\Cref{disper} shows this set of $\omega_\trp$ found at each
$\kappa$.
The gap between $\omega_\trp(\kappa)$ and $|\kappa|$ indicates the
strength of trapping (rapidity of evanescent decay as $x_2\to\infty$).
The most trapped mode is at $\kappa = \pm\pi$, known as an
``optical mode'', and has $\omega_\trp \approx 2.551$ for the $d=1$
staircase.
As $\kappa \to 0$ we see trapping becoming arbitrarily weak.

With a zero of the determinant found,
the null vector $\sigma$ obeying $(I-2D^T)\sigma=0$ is found
via an SVD, then $\phi = \S\sigma$ gives the mode.
Numerically the lattice-sum expansion may only be used up to a height
around $d/2$ above the origin, so as before one must switch
to Fourier series evaluation above this.
\Cref{modes} shows the real part of example modes.

\subsection{Ray model for time-domain chirp response at \textit{El Castillo}}
\label{ray}

The above trapped mode dispersion curve provides a simple
explanation of the acoustic ``raindrop'' effect on the long
stone staircases at the \emph{El Castillo} pyramid at Chichen-Itza.
\footnote{We are very grateful to Eric Heller for suggesting this
explanation; it is also described by Heller in \cite[p.~162]{hellerbook}.}
Impulsive sources of sound, such as footsteps, are reported
to sound like short chirps (frequency rising vs time) when heard from
distances far up or down the staircase.
Pending a full numerical investigation of the time-domain
solution to \cref{WE,WEBC}, we use a simple ray model
for dispersive propagation \cite[Sec.~2.6]{buhler}.
Both source and receiver are assumed to be
close enough to the surface to couple well to
the modes $\phi$ at all wavenumbers (this is true for a footstep,
less obviously so for a standing listener).
Absorption and modeling of amplitudes are ignored.
Impulsive excitation in \cref{WE} is assumed to excite all frequencies.
Each frequency $\omega$ below the maximum of $\omega_\trp$
is partly trapped (see \cref{power-in-trapped}),
and this component propagates along the
staircase at its \emph{group velocity}
\be
v_g = \frac{\d \omega_\trp}{\d \kappa},
\label{vg}
\ee
at the appropriate $\kappa := \kappa_\trp$ such that $\omega_\trp(\kappa) = \omega$.
We plot $v_g$ vs $\kappa$ on the right of \cref{disper},
by differentiating the interpolant of $\omega_\trp$ on the $\kappa$-grid.
The frequency's arrival time a distance $D$ along the staircase
is thus $t(\omega) = D / v_g(\omega)$.
Finally, inverting this last relationship gives the frequency $\omega$
heard at each time $t$ after the impulse.

We now insert physical parameters for \emph{El Castillo}.
The nondimensionalized speeds used in the rest of the paper must
be multiplied the sound speed $c=343 $ m/s.
According to \cite{declercq2004theoretical} the stairs at this pyramid
have depth equal to height equal at $q=0.263$ m, i.e.\ a $\pi/4$ staircase
with period $d = \sqrt{2} q \approx 0.372$ m.
The maximum trapped mode frequency
is thus $(c/d) 2.551 \approx 374$ Hz;
this is the highest frequency explainable in the model.
\footnote{The frequency of a free-space plane wave traveling along the staircase
with wavenumber at the Brillouin zone edge $\kappa = \pi/d$
is $c/2d \approx 461 $ Hz. Neither of these frequencies appear in \cite{calleja}, although a ``raindrop frequency'' of 307.8 Hz is mentioned.}
Each staircase has 91 steps, giving $D \approx 34$ m.

The resulting predicted frequency vs arrival time
is shown in \cref{chirp}.
The first arrivals are the lowest frequencies; for these the dispersion
curve is almost that of free air ($v_g \approx c$) so they start
arriving immediately after any direct (non-trapped) radiation.
Most of the frequency ``chirp'' occurs during the first 0.2 s after first
arrival, an interval containing of order 50 cycles, and thus
plenty to detect the upwards frequency trend.
A long ``bell-like'' tail, asymptoting to the maximum 374 Hz, is expected,
associated with the slowly-propagating optical modes
as shown on the left of \cref{modes}.

\begin{rmk}
  We provide a link to a WAV file simulated to match the above chirp
    prediction at this URL: \url{https://doi.org/10.5281/zenodo.10005461}.
  The authors do not have access to audio recordings of footsteps
  on this staircase, and would welcome the chance to validate the
  predictions.
\end{rmk}

\section{Scattering from a single point source \label{singleptsrc}}

As laid out in \cref{intro}, we simulate scattering from a single point source
by integrating over the solutions from a quasiperiodic array of point sources
with different $\kappa$ values. Let $u_{\kappa,i}(\xx) = \phip_{,\kappa}(\xx, \xx_0)$, where the
subscript $\kappa$ refers to having fixed $\kappa$ and $x_0$ is a given source
position within the central unit cell. Then from \cref{asm-setup-per,asm-setup-single} it follows that
\be
u_i(\xx) = \Phi(\xx, \xx_0) = \frac{1}{2\pi}\int_{-\pi}^{\pi} u_{\kappa,i}(\xx) \d \kappa,
\label{singlesrc-gf-recon}
\ee
and, with $u_{\kappa}(\xx)$ referring to the total field due to the above incident field $u_{\kappa, i}(\xx)$, the array scanning integral is
\be
u(\xx) = \frac{1}{2\pi}\int_{-\pi}^{\pi} u_{\kappa}(\xx) \d \kappa = \frac{1}{2\pi}\int_{-\pi}^{\pi} (u_{\kappa, i}(\xx) + u_{\kappa, s}(\xx))\d \kappa,
\label{asm-ux}
\ee
where $u_{\kappa, s}$ is the scattered field associated with $u_{\kappa, i}$.

To achieve high-order accuracy one must
understand the complex $\kappa$ plane singularities of the above integrand.
For example, in \cref{complexkappaplane-path}
we plot the integrand in \eqref{asm-ux}
for a target point $\xx = (0.22, -0.16)$,
frequency $\omega = 2.4$, and source point
$x_0 = (-0.2, 0.1)$. Each side of the stair boundary was split into $4$ panels
initially, then refined $5$ times with $r = 3$, until the total number of
quadrature panels over the boundary was $Q = 28$, each with $P = 16$ nodes.
This gives around $7$ accurate digits according to the
convergence test in \cref{flux-conv-perisrc-neu}. The two key features in the figure
are the branch points due to Wood anomalies at $\kappa = \pm \omega$, and
two poles at $\kappa = \pm\kappa_{\mathrm{tr}}$.
They are ordered such that $\omega \leq \kappa_{\mathrm{tr}} \leq \pi$. The branch
cuts associated with the branch points are indicated by dashed lines. The direction of the branch cuts can be
chosen by altering the integration path in \cref{sn-int} (by choosing $\tau(t)$), but as it will become clear
later, it is convenient to choose the cuts to lie in the lower and upper
half-plane for negative and positive $\kappa$, respectively. If $\omega$ were
larger than $\omega_c$,
where no trapped modes exist for any $\kappa$, there would be no poles.
In the limits $\omega \to 0$ and $\omega \to \omega_c$, the
two poles coalesce at $\kappa_{\mathrm{tr}} \to 0$ and $\kappa_{\mathrm{tr}} \to \pm \pi$ respectively.

\begin{figure}[tb]
\centering
    \includegraphics[trim = {1.2cm, 0cm, 2.7cm, 0cm}, clip, width = \textwidth]{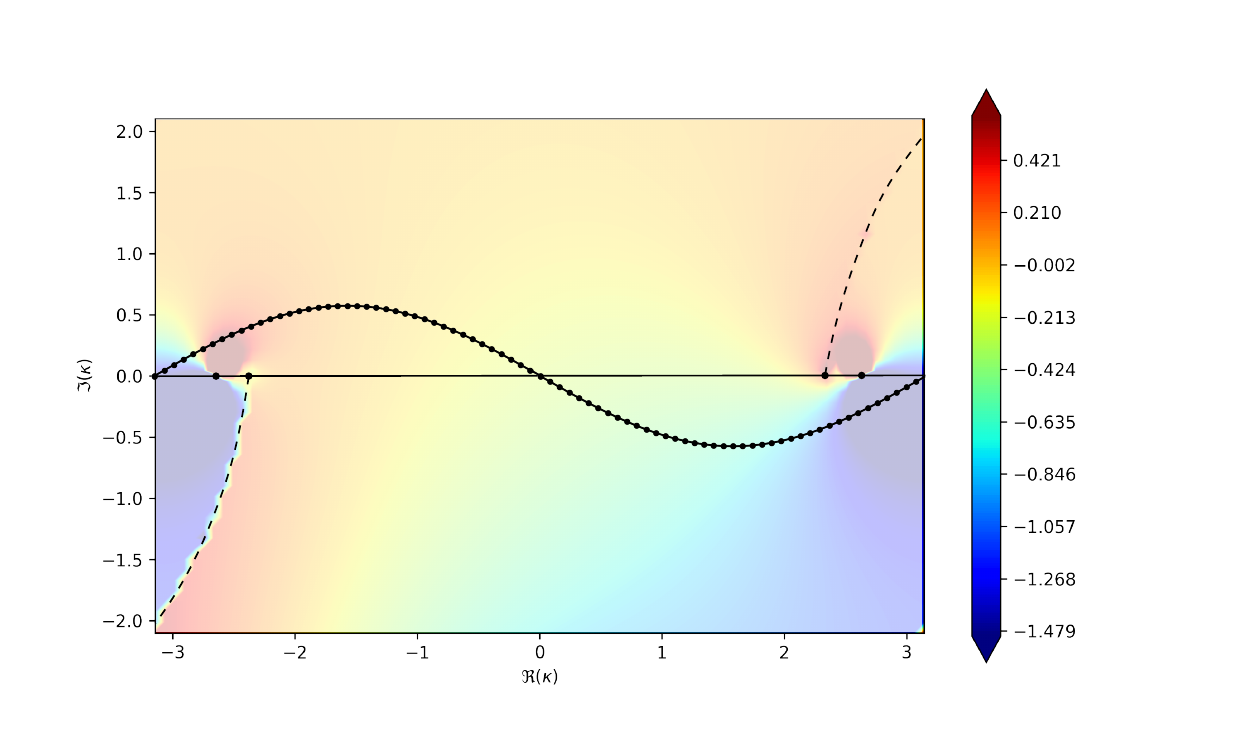}
    \caption{Real part of the integrand of the array scanning integral \cref{asm-ux} plotted in
    the complex $\kappa$-plane, where the field is evaluated at a single
    target, $\xx = (0.22, -0.16)$. Plotted over this are the branch points and cuts
    (dots with dotted lines), poles (dots), and the contour deformation used (solid black
    line, with an example set of 80 quadrature nodes plotted on top). \label{complexkappaplane-path}}
\end{figure}

The branch points and poles in \cref{asm-ux} require careful treatment. If
$\omega > \omega_c$, the branch points can be integrated over without leaving
the real axis, due to Wood anomalies being known to be square-root
singularities
\cite{fano1941theory,hessel1965new,bolotovskii1968threshold,wojcik2021universal}.
They may then be dealt with using special Gaussian quadrature, \eg
\cite{huybrechs2009generalized}. With the emergence of the trapped mode poles
at $\omega \leq \omega_c$, options include singularity extraction
\cite{rana1981current}, and contour deformation
\cite{he2007radiation,lovat2011dipole}. We opt for the latter, deforming the
contour away from the real axis and the branch cuts, as shown by the solid black line in
\cref{complexkappaplane-path}.

Note that in order for the solution to correspond to outgoing waves, the
integration contour has to be deformed into the upper half-plane for negative
$\kappa$ and the lower half-plane for positive $\kappa$. This idea was
introduced in \cite{v1905reflexion} and termed the \emph{limiting absorption
principle} by \cite{sveshnikov1950radiation}. The direction of the branch cuts
was chosen with this in mind, following \cite{ruming21}.

We choose a sinusoidal deformation contour parameterized by $k:= \re \kappa$,
\be
\kappa(k) = k -iA\sin(k), \label{sine-contour-param}
\ee
which is then discretized using the periodic trapezoidal rule with $p_{\mathrm{asm}}$ nodes.
To get an idea of what accuracy a given $A$ and
$p_{\mathrm{asm}}$ yields, we use \cref{singlesrc-gf-recon} as a test
problem, and investigate how accurately $\Phi$ can be reconstructed from
$\phip_{,\kappa}$ in \cref{conv-phi-asm}. The top part of the figure shows
the real part of total solution $u(\xx)$, computed to $7$ digits of accuracy at $\omega = 2.4$, with a point source at $\xx_0 = (-0.2, 0.1)$. The bottom part of
\cref{conv-phi-asm} shows
that the exponential convergence rate is sensitive to the exact value of the amplitude, and that values
between $0.5$ and $2$ are optimal.
In this range, $p_{\mathrm{asm}} \approx 100$ quadrature nodes are
enough to get an answer accurate to machine precision.

\begin{figure}[tb]
\centering
    \subfloat{\includegraphics[width = \textwidth]{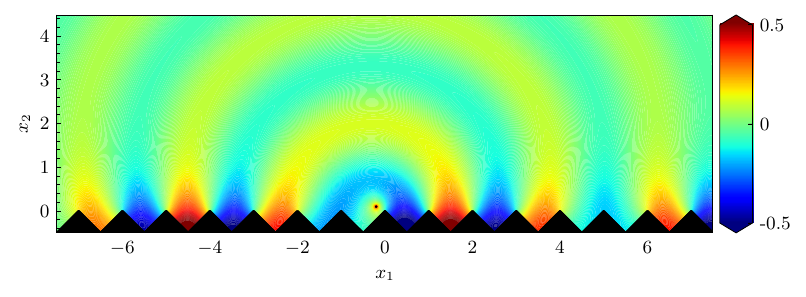}}\\
    \subfloat{\includegraphics[width = \textwidth]{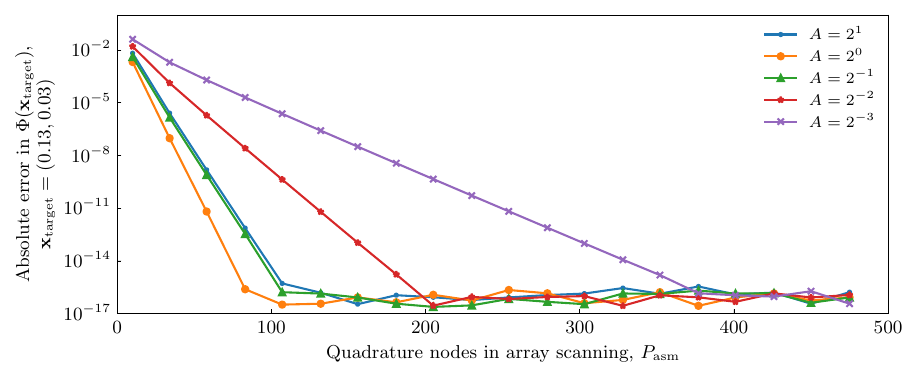}}
    \caption{Top: Real part of total acoustic pressure field from a point source at $\xx_0 = (-0.2, 0.1)$, radiating with frequency $\omega = 2.4$, accurate to $7$ digits. Bottom: convergence test of the reconstruction of
    $\Phi(\xx_{\mathrm{target}})$, at $\xx_{\mathrm{target}} = (0.13, 0.03)$, with
    the array scanning method. The sinusoidal contour shown in
    \cref{complexkappaplane-path} was taken, with various amplitudes $A$,
    and discretized with $p_{\mathrm{asm}}$ trapezoidal quadrature nodes.
    \label{conv-phi-asm}}
\end{figure}

So far we only considered computation of the solution inside the central unit
cell (and above it), $-\tfrac{d}{2} \leq x_1 \leq \tfrac{d}{2}$. Using the
quasiperiodicity property, the array scanning integral \cref{asm-ux} for a target
position in the $n$th unit cell becomes
\be
u^n(\xx) := u(x_1 + nd, x_2) = \frac{1}{2\pi}\int_{-\pi}^{\pi}u_{\kappa}(\xx)e^{ind\kappa}\d \kappa.
\label{asm-ux-n}
\ee
\Cref{asm-integrand-nlarge} shows the array scanning integrand at a target $n =
5$ unit cells away from the center, to be compared to
\cref{complexkappaplane-path}. The exponential
factor causes oscillations in the real and exponential growth in the imaginary
direction in the lower (upper) half-plane for positive (negative) $n$. The exponential growth limits the
amount of contour deformation away from the real axis, and both the growth and
oscillations demand more quadrature points to maintain constant accuracy as $n$
increases.
Thus as $n$ grows, quadrature along the contour becomes a less practical
method to evaluate the solution.

However, in the limit $n \to \pm \infty$, the solution may be
computed easily in a different manner.
Consider first the case $n\to+\infty$.
We deform the sinusoidal contour to the contour shown in orange
in \cref{asm-integrand-nlarge};
if there is a trapped mode (pole) at $\kappa_\trp$ then this
deformation introduces a correction by the residue of that pole.
The contribution of the orange contour vanishes in the limit $n \to +\infty$
by standard Jordan's lemma type arguments.
Namely, the term $e^{ind \kappa}$ vanishes
for any $\im \kappa > 0$, making the contributions of the segments parallel to the branch cut, and the horizontal segments, zero. The other vertical segments cancel by periodicity.
Taking the radius of the ``keyhole'' around the branch point to zero,
its contribution vanishes since the integrand involves powers of at least $-1/2$.
By Cauchy's theorem, the value of the array scanning
integral is thus only the residue at $\kappa_\trp$, if such a pole exists,
and zero otherwise.
Considering now only the case with such a residue,
since the phase of $u$ changes with $n$, one has to take care incorporating
the correct phase in order for a limit to exist:
\begin{align}
  u^{+,\infty}(x_1,x_2) &:=
  \lim_{n \to \infty} e^{-ind\kappa_\trp} u(x_1 + nd, x_2) = \lim_{n \to \infty} \frac{1}{2\pi} \int_{-\pi}^{\pi} u_{\kappa}(x_1, x_2)e^{in(\kappa-\kappa_\trp)d}\d \kappa \nonumber \\
    &=  i \underset{\kappa = \kappa_\trp}{\Res} u_{\kappa, s}(x_1, x_2).
    \label{uinfty}
\end{align}
In the case $n\to-\infty$, we deform the sinusoidal contour into the lower half-plane, as shown in blue in \cref{asm-integrand-nlarge}. By the same arguments, if a trapped mode exists, the solution is
\be
  u^{-,\infty}(x_1,x_2) := \lim_{n \to -\infty} e^{-ind\kappa_\trp} u(x_1 + nd, x_2) =  -i \underset{\kappa = -\kappa_\trp}{\Res} u_{\kappa, s}(x_1, x_2),
    \label{uminfty}
\ee
and zero otherwise.

Finally, we explain how we
numerically extract the residue at $\kappa = \pm \kappa_\trp$,
which can be viewed as extracting the left- and right-going
trapped mode amplitudes.
We use Cauchy's theorem once more, and integrate on
a circular contour enclosing the relevant pole, applying
a trapezoidal rule with $p_\res$ nodes.
The radius of this circle, $r_\res$, is chosen so that the amplitude of the
integrand along the contour stays as close to $1$ as possible, thus avoiding
loss of accuracy due to catastrophic cancellation. The contour also needs to
lie on one sheet, and therefore cannot cross a branch cut. The radius $r_\res$ is therefore determined by the distance to the nearest branch point or the distance to the nearest pole, whichever is smaller.
We discuss the $\kappa = +\kappa_\trp$ case and use the same parameters for negative $\kappa$. By inspection of the dispersion relation in
\cref{disper}, it is clear that for all but $\kappa_\trp \to \pi$, the distance to the nearest branch point, at $\kappa = \omega$, is smaller. We use
\be
r_\res = \min \left(0.5\cdot |\kappa_\trp - \omega|, |\pi - \kappa_\trp|\right) \label{r-res}
\ee
to determine the radius, and $p_\res = 64$ nodes, after convergence testing by doubling $p_\res$.

The main claim of the previous section is a direct consequence of the above
result: at $\omega < \omega_c$ the only contributor to the field infinitely far
away from the source is the trapped mode at the given $\omega$. At $\omega >
\omega_c$, where no trapped modes exist, the field in this limit is zero, since
the contour in the complex $\kappa$ plane contains no poles.
As we have seen in
\cref{trapped-modes}, trapped modes at different frequencies have different vertical
decay lengths, and propagate at different speeds. It is therefore of interest to
compute how much of the power injected into the system is transported away in
trapped modes, as opposed to radiated vertically,
as a function of frequency---this
is the subject of the next section. 

\begin{figure}[tb]
\centering
    \includegraphics[trim = {1.2cm, 0cm, 2.7cm, 0cm}, clip, width = \textwidth]{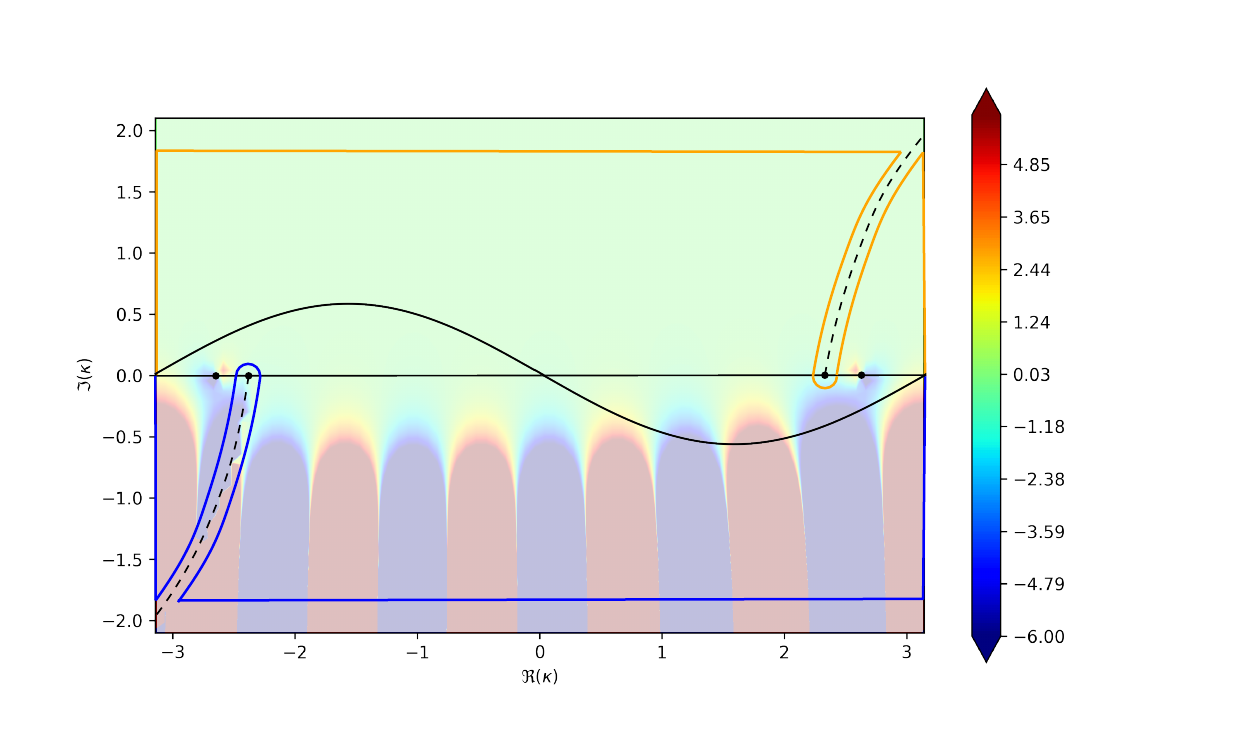}
    \caption{Real part of the array scanning integrand \cref{asm-ux-n} with the
    target point $n = 5$ unit cells away from the source, to be compared with
    \cref{complexkappaplane-path}. Identical branch cuts, poles, and
    contour deformation are shown in black. In orange and blue we plot the contours
    we use to derive the field in the limit of $n \to \pm \infty$,
    respectively. \label{asm-integrand-nlarge}}
\end{figure}

\section{Extracting the asymptotically trapped power \label{power-in-trapped}}

In this section we find the total power injected into the system by a single
point source, and compute the fractional power carried to infinity in trapped modes, as a
function of frequency. The source $\xx_0$ is assumed to be inside the central unit cell, \ie, $-\tfrac{d}{2} < x_1^0 < \tfrac{d}{2}$.

The total power radiated by a single point source
can be derived by taking the
total flux \cref{net-flux-formula} exiting a circle of radius $r$ (denoted
$\Gamma_r$) centered on the source, and then taking the limit $r \to 0$:
\be
P_{\mathrm{tot}} = \lim_{r \to 0} \left(\Im \int_{\Gamma_r} (\bar{u}_s + \bar{u}_i)\partial_n(u_s + u_i) \, \d s \right).
\ee
Out of the four terms, only two contribute: since $u_s$ and its derivative is
finite everywhere, the integral of $u_s\partial_n u_s$ vanishes; and $u_i$ only
grows like $\ln r$ as $r \to 0$, therefore its contribution is $\lim_{r \to 0} r\ln r =
0$. Using the asymptotic form of $H_0^{(1)}(\omega r)$ as $r \to 0$ \cite[Ch.\ 10.7]{dlmf}, it can be shown that the remaining terms add up to
\be
P_{\mathrm{tot}} = \frac{1}{4} + \Im(u_s(\xx_0)).
\ee
Thus the radiated power is influenced by the scattered wave at the emission point.

For each value of $\omega$, we use the strategy outlined in \cref{singleptsrc}
to compute $u_s(\xx_0)$, \ie integrate along a deformed array scanning contour
in $\kappa$ using the
trapezoidal rule. However, in the limits $\omega \to 0$
(where $\kappa_\trp\to0$), or $\omega \to \omega_c$
(the cutoff frequency, where $\kappa_\trp\to \pi$),
the contour must pass between two coalescing poles.
This necessitates an increase in
quadrature node density in the section of the contour closest to the
poles. In these limits, it is therefore inefficient to evenly space the nodes
in $\re \kappa$. We instead follow \cite[Sec.~3]{barnettexpgrading} and define an
exponentially graded reparameterization of the real part of $\kappa$,
via a periodic map $\theta: [-\pi,\pi)\to[-\pi,\pi)$, namely
\be
\kappa(k) = \theta(k) - iA\sin(\theta(k)), \quad \text{where } \theta'(k) = \begin{cases}
    \alpha \cosh \left(b\sin \frac{k}{2}\right) \quad \text{if } \kappa_\trp \to \pi,\\
    \alpha \cosh \left(b\sin \frac{k - \pi}{2}\right) \quad \text{if } \kappa_\trp \to 0,
\end{cases}
\label{expgrading}
\ee
which bunches quadrature nodes that are evenly spaced in $k$
% AB: but $k = \Re\kappa$ is no longer true
by a factor of order $e^b$ close to either $\pi$ or $0$.
The normalization $\alpha$ is determined
numerically by requiring $\int_0^{2\pi}\theta'(k)\d k = 2\pi$.
To achieve uniform accuracy as $\kappa_\trp \to 0$ or $\kappa_\trp \to \pi$,
one needs to update $b$ as $b = c_1 + c_1\log \kappa_\trp$ and $b = c_3 +
c_2\log(\pi - \kappa_\trp)$, respectively, where $c_1$--$c_3$ are constants to
be determined. Here, we consider $\kappa_\trp, \pi - \kappa_\trp > 0.05$, for
which we find (by doubling $p_{\mathrm{asm}}$ to gain an upper bound on the
error) that setting $b = 5$, $p_{\mathrm{asm}} = 60$ is sufficient to obtain a
value for $P_{\mathrm{tot}}$ accurate to 8 digits. We therefore use these
settings if $\omega \lesssim 0.5$ or $2.6 \lesssim \omega$, and $b = 0$ (evenly
spaced nodes in $\Re \kappa$) otherwise.
This parameter combination ensures that the nearest pole is at
least 5 quadrature node spacing's distance away from the contour.

\begin{table*}[tb]
    \footnotesize
    \centering
    \renewcommand{\arraystretch}{1.2}
\begin{tabular}{llr}
\cline{1-3}
    
Parameter          & Description                                                                                            & Value                                                 \\ \cline{1-3}
$A$ & Amplitude of sinusoidal array scanning contour in \cref{expgrading}           & 1.0                                                  \\
$p_{\mathrm{asm}}$ & Number of trapezoidal nodes along sinusoidal array scanning contour                                    & 60                                                      \\
$b$ & Exponential grading parameter along sinusoidal array scanning contour                                    & 0.0 if $0.5 \lesssim \omega \lesssim 2.6$, 5.0 otherwise                                                        \\
$r_\res$           & Radius of circular contour for residual calculation                                                    & given by \cref{r-res}                      \\
$p_\res$           & Number of trapezoidal nodes along circular residual contour                                            & 64                                                      \\
    $\epsilon$            & Tolerance for determining the upper limit \cref{trapped-int-upperlim} of trapped power integral \cref{trapped-pow-integral}                     & $10^{-8}$  \\
    $p_\trp$           & Number of Gauss--Legendre nodes in trapped power integral \cref{trapped-pow-integral} &                                             128          \\ \cline{1-3}
\end{tabular}
    \caption{Parameters used in computing the fractional power transported by trapped modes, chosen so that the total power $P_{\mathrm{tot}}$ and power in trapped modes $P^{\pm}_\trp$ is accurate to $8$ digits for all $\kappa_\trp$ considered, \ie, $\kappa_\trp, \kappa_\trp - \pi > 0.05$. \label{trap-param-summary}}
\end{table*}

As shown in \cref{singleptsrc}, the only contributors to the field infinitely far away
from the source along the surface are trapped modes.
Let $\Gamma_\infty$ be the semi-infinite vertical
line segment with a given $x_1$ coordinate, starting on the surface.
We can use \cref{uinfty} to reconstruct $u^{+,\infty}$ or $u^{-,\infty}$
on this line segment.
Then the power in the left-going ($-$) or right-going ($+$)
trapped modes is
\be
P^\pm_\trp = \pm
\Im \left( \int_{\Gamma_\infty} \bar{u}^{\pm,\infty} \partial_{x_1} u^{\pm,\infty} \d s \right).
\label{trapped-pow-integral}
\ee
Note that this expression must be independent of the choice of
horizontal position $x_1$, by power conservation.
Furthermore, the phase introduced in \cref{uinfty} cancels.
Given any $x_1$, we approximate \eqref{trapped-pow-integral}
by a quadrature rule from the surface point $x_2^0$ to an
upper limit 
$x_2^1$, chosen such that the given trapped mode has sufficiently decayed.
Based on the vertical
mode intensity decay rate (twice the amplitude decay rate in \eqref{kn})
we thus set $x_2^1$ to
\be
x_2^1(\kappa) := x_2^0 + \frac{\log(1/\epsilon)}{2\sqrt{\kappa^2 - \omega^2}}, \label{trapped-int-upperlim}
\ee
where $\epsilon>0$ is a desired error tolerance.

While the choice of $x_1$ is immaterial mathematically,
one choice is much more convenient numerically:
the corner (trough) at $x_1 = \pm\tfrac{d}{2}$ is best, for the following reason.
Firstly, $x_1$ passing through a corner is to be preferred to any other
part of $\Gamma$,
since the panels discretizing $\Gamma$ are already geometrically
refined towards corners, allowing accurate evaluation
arbitrarily close to the corner via plain quadratures.
(In contrast, were $x_1$ to intersect any flat part of $\Gamma$,
a special close-evaluation quadrature would be needed at points
closer than one panel-size from $\Gamma$.)
Secondly, to decide whether the upper ($3\pi/2$ angle at $x_1=0$) or lower ($\pi/2$ angle at $x_1=\pm d/2$) corner is to be preferred,
one expands the total field in terms of Bessel functions
around either point and imposes Neumann boundary conditions.
This shows that only the lower corner has a regular (hence analytic) expansion
involving even powers.
The potential at the upper corner is nonanalytic since it involves
powers that are multiples of $2/3$.
The above then
implies that Gauss--Legendre quadrature with $p_\trp$ nodes is high-order accurate for the integral
\cref{trapped-pow-integral} along the line $x_1 = \pm d/2$, so this
is what we use, with $p_\trp = 128$.

\Cref{flux-distr} shows the power balance as a function of $\kappa$ on the left
and $\omega$ on the right, with power tolerance $\epsilon=10^{-8}$. For
reproducibility, we summarize the parameters used in \cref{trap-param-summary}.
Since the fluxes calculated here are unitless, we
show the flux in trapped modes $P^{+}_\trp + P^{-}_\trp$, the total flux, and the fraction of the flux
carried in trapped modes all in one plot. As expected from their large vertical
decay length, the least trapped modes at $\omega$, $\kappa \to 0$ carry the
least flux. The highest-frequency trapped modes carry a fraction of the total
power that approaches $1$, and the increase of this fraction with frequency or
Bloch wavenumber is close to linear until an abrupt drop to zero at $\omega >
\omega_c$. These ``critically trapped'' acoustic modes are the most
efficient both
at sucking power out of the source (given a fixed source amplitude),
and at trapping this input power at the surface.

\begin{figure}[tb]
\centering
\includegraphics[width = \textwidth]{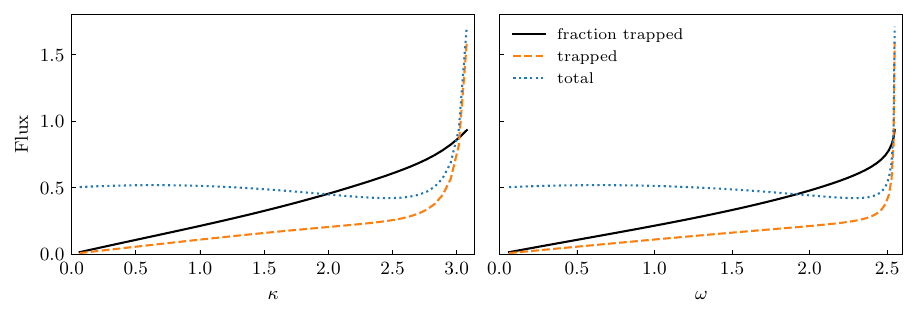}
    \caption{Left: Total flux (unitless) in the system, flux in trapped modes, and fraction of flux in trapped modes, as a function of Bloch wavenumber $\kappa$. Right: The same quantities plotted against the frequency $\omega$ of the Helmholtz equation.  \label{flux-distr}}
\end{figure}

\section{Conclusions and future work \label{future}}

We described a boundary integral equation method for solving the
$2$-dimensional, constant-coefficient Helmholtz equation on an exterior domain
outside of an infinite, periodic boundary with corners. We did this in the
context of acoustic scattering, with a particular emphasis on trapped modes.
Using corner-refined Nystr\"om quadrature on the boundary, we built
a dense direct solver for the quasiperiodic problem.
By integrating over the quasiperiodicity parameter, a method known
as array scanning, we computed the scattering solution from a single point
source, and extracted the limit of the acoustic pressure field infinitely far
away from the source. Obtaining high-order accuracy required a detailed
understanding of poles and branch-cuts, due to trapped modes and Wood anomalies
respectively. To this end we proposed a complex contour deformation and nonuniform
reparametrization for an efficient quadrature.
We proposed a residue method to extract the amplitudes of left- and right-going
trapped surface waves, and used this
to study the fraction of injected acoustic power
that ends up in trapped modes, as a function of frequency.

We show that the trapped modes (quasiperiodic eigenfunctions)
map precisely to roots of a Fredholm determinant---see Theorem~\ref{SLP}.
By applying Nystr\"om quadrature to this, we
compute the trapped mode dispersion relation and group velocity.
A simple ray model then allowed us to
predict frequency vs arrival time in the ``chirp'' phenomenon observed at \textit{El Castillo}
at Chichen Itza, and other similar acoustics recorded at step-temples.

Several questions remain for staircase acoustics:
are there source locations close to the surface which excite trapped modes less, or excite the left- and right-going
modes asymmetrically? Similarly, it would be of interest to extend the analysis
to a wider range of periodic boundaries that possess asymmetry.

Although dense direct linear algebra as presented here was adequate
for the studied geometry and accuracies, more complex geometries
requiring more than $10^4$ discretization nodes
would benefit from the use of iterative solution with FMM acceleration.
The extension to 3D doubly-periodic structures to high-order accuracy
poses an interesting challenge, because there the set
of Wood-anomaly on-surface wavevectors form curves.

\section*{Acknowledgments}

We are indebted to Eric J.\ Heller for suggesting that the
raindrop effect is due to the dispersion relation of
trapped evanescent waves; this motivated developing the methods
presented here.
We also thank Charlie Epstein, Manas Rachh, Leslie Greengard, and Simon Chandler-Wilde
for their helpful comments and input.
The Flatiron Institute is a division of the Simons Foundation.

\bibliographystyle{elsarticle-num}
\bibliography{refs}
\end{document}